\documentclass[11pt]{article}

\usepackage{amsmath,amssymb,amsthm}
\usepackage[margin=1in]{geometry}
\usepackage{hyperref}
\usepackage{color}
\usepackage{enumerate}
\usepackage{todonotes}

\title{A weighted cycle-localization inequality}
\author{Jiangdong Ai\thanks{School of Mathematical Sciences and LPMC, Nankai University. {\tt jd@nankai.edu.cn}. Funded by the National Natural Science Foundation of China (No.12522117, No.12401456), the Natural Science Foundation of Tianjin (No.24JCQNJC01960) and Fundamental and Interdisciplinary Disciplines Breakthrough Plan of the Ministry of Education of China (JYB2025XDXM207).}
\hspace{2mm}
Bin Chen\thanks{School of Mathematics and Statistics, Fuzhou University, Fujian, China. Email: {\tt cbfzu03@163.com}. Funded by the National Natural Science Foundation of China (No.12501473, No.12471336).}
\hspace{2mm}
Ming Chen\thanks{School of Mathematics and Statistics, Jiangsu Normal University, Xuzhou, China. Email: {\tt chenming314@jsnu.edu.cn}. Funded by National Key Research and Development Program of China (No.2024YFA1013900), National Natural Science Foundation of China (No.12501483), Basic Research Program of Jiangsu (No.BK20251044), and Natural Science Foundation of the Jiangsu Higher Education Institutions of China (No.25KJB110003).}
\hspace{2mm}  Tianxiao Zhao\thanks{School of Mathematics, Harbin Institute of Technology, Harbin, China. {\tt zhaotianxiao@hit.edu.cn}.  Funded by the National Natural Science Foundation of China (No.12501475).}}
\date{March 2026}

\newtheorem{theorem}{Theorem}
\newtheorem{lemma}[theorem]{Lemma}
\newtheorem{question}{Question}
\newtheorem{proposition}[theorem]{Proposition}
\newtheorem{corollary}[theorem]{Corollary}

\theoremstyle{definition}
\newtheorem{definition}[theorem]{Definition}
\newtheorem{remark}[theorem]{Remark}

\begin{document}
\maketitle

\begin{abstract}

In 1959, Erd\H{o}s and Gallai showed that every $2$-connected graph $G$ contains a cycle of length at least $\frac{2|E(G)|}{|V(G)|-1}$. This result was subsequently extended to weighted graphs by Bondy and Fan in 1991. A natural local variant of this problem arises by considering, for each edge $e\in E(G)$, the quantity $c(e)$, defined as the length of the longest cycle in $G$ containing $e$ (with $c(e)=2$ if $e$ is a bridge). Zhao and Zhang recently proved that for every graph $G$ on $n$ vertices satisfies
$\sum_{e\in E(G)}\frac{1}{c(e)}\le \frac{n-1}{2}.$

In this note, we establish a weighted generalization of this inequality. For a weighted graph $(G,w)$ with positive edge weights, let $C_w(e)$ denote the maximum weight of a cycle containing $e$ (setting $C_w(e)=2w(e)$ if $e$ is a bridge). We prove that
$$
\sum_{e\in E(G)}\frac{w(e)}{C_w(e)}\le \frac{n-1}{2}.
$$
Our result can be viewed as a weighted local analogue of the Bondy-Fan theorem, thereby establishing a correspondence between the global and local perspectives. Furthermore, we present a broad class of graphs attaining equality and derive necessary conditions for equality.




\end{abstract}

\section{Introduction}
Determining the length of the longest cycle is one of the most important lines of research in graph theory.
From a global perspective (i.e., the number of edges), Erd\H{o}s and Gallai~\cite{EG} obtained the following result.
\begin{theorem}[\cite{EG}]\label{eg}
    Every $2$-edge-connected  graph with $n$ vertices and $m$ edges contains a cycle of length at least $\frac{2m}{n-1}$.
\end{theorem}
Let $G$ be a weighted graph if each of edge $e\in E(G)$ is assigned a nonnegative number $\phi(e)$.
When $\phi(e)=1$ for each $e\in E(G)$, $G$ is a simple graph.
For any subgraph $H\subseteq G$, the \emph{weight} of $H$ is defined by $$\phi(H)=\sum\limits_{e\in E(H)}\phi(e).$$
In 1991, Bondy and Fan \cite{BF} resolved a conjecture from \cite{BF1}, thereby generalizing Theorem \ref{eg} to weighted graphs.
Concretely, their result reads as follows.
\begin{theorem}[\cite{BF}]\label{bf}
    Let $G$ be a 2-edge-connected weighted graph on $n$ vertices. Then $G$ contains a cycle of weight at least $\frac{2\phi(G)}{n-1}$.
\end{theorem}
In the same paper, Bondy and Fan \cite{BF} further characterized the 2-edge-connected weighted graphs on $n$ vertices that contain no cycle of weight more than $\frac{2\phi(G)}{n-1}$.

The aforementioned two results suggest that a suitably chosen global condition can yield a good lower bound on the length of the longest cycle in graphs.
A natural question then arises:
\begin{center}
    \emph{Is there a local condition that yields a good lower bound on the length of the longest cycle in graphs?}
\end{center}
Recently, Zhao and Zhang \cite{ZZ} introduced a \emph{localized} perspective for long cycles.
For a simple graph $G$ and any $e\in E(G)$, let $c(e)$ be the length of a longest cycle containing $e$ (and set $c(e)=2$ if $e$ is a bridge).
They provided the following result.
\begin{theorem}[\cite{ZZ}]\label{ZZ}
    Let $G$ be a graph on $n$ vertices. Then 
    \begin{equation}\label{eq:ZZ}
\sum_{e\in E(G)}\frac{1}{c(e)}\le \frac{n-1}{2}, 
\end{equation}
with equality if and only if $G$ is connected, and each block of $G$ is a clique.
\end{theorem}

Here, we want to remark that Theorem \ref{ZZ} can be viewed as a generalization of Theorem \ref{eg} as illustrated by the following inequalities: \begin{equation}\label{eq2}
\sum_{e\in E(G)}\frac{1}{L}\leq\sum_{e\in E(G)}\frac{1}{c(e)}\le \frac{n-1}{2}
\Longrightarrow{\frac{e(G)}{L}\leq \frac{n-1}{2}}\Longrightarrow\frac{2e(G)}{n-1}\leq L,
\end{equation}
where $L$ is denoted to be the length of the longest cycle in $G$.

Very recently, Li and Ning~\cite{LN} observed that \eqref{eq:ZZ} admits an elegant proof by applying Theorem \ref{bf} to an \emph{auxiliary} weighting of the form $e\mapsto \frac{1}{c(e)}$.
This is already a useful viewpoint, but it remains essentially unweighted: the input graph has no edge-weights, and the weighting is introduced only as a proof device.

Inspired by Theorems \ref{eg} and \ref{bf}, it is natural to ask the following question.
\begin{question}
    Can we generalize Theorem \ref{ZZ} to weighted graphs?
\end{question}
In our view, the key to this question is how to choose a suitable local parameter for each edge, and such a choice is far from obvious.
A natural requirement of this parameter is that the inequality should depend on the given weights $w$ and should localize \emph{the heaviest cycle through each edge}.
This leads to the parameter $C_w(e)$, which means the maximum weight of a cycle containing $e$.
The point is that $C_w(e)$ is defined using the same weighting $w$ that we want to control, and it varies edge-by-edge in a highly non-linear way.

The central observation of this note is that the self-normalization
$$
\phi(e):=\frac{w(e)}{C_w(e)}
$$
forces \emph{every} cycle $C$ to satisfy $\phi(C)\le 1$, despite the denominators $C_w(e)$ depending on global maximal.
This makes the weighted problem tractable and produces a sharp analogue of Theorem \ref{ZZ}.
In particular, the weighted equality theory is strictly richer than the unweighted one: even on $K_n$, equality can hold for many non-uniform weightings (see Proposition~\ref{prop:induced}).
At the same time, equality is not ``free'': on $K_r$ with $r\ge 4$ we show that equality forces the original weights themselves to be vertex-induced (see Proposition~\ref{prop:Kchar}).

Before stating our main result, we formally introduce several necessary definitions.
Throughout, $G$ is a finite simple graph on $n$ vertices and $w:E(G)\to \mathbb{R}_{>0}$.
For a subgraph $H\subseteq G$, write $w(H)=\sum_{e\in E(H)} w(e)$.

\begin{definition}
For any $e\in E(G)$, let
$$
C_w(e):=
\begin{cases}
\max\{w(C):\ C\subseteq G\text{ is a cycle and }e\in E(C)\}, & \text{if $e$ lies on a cycle},\\[2pt]
2w(e), & \text{if $e$ is a bridge}.
\end{cases}
$$
\end{definition}
Note that $C_w(e)$ generalizes $c(e)$ from \cite{ZZ} by Zhao and Zhang in the sense that the latter corresponds to the special case where $w\equiv 1$.


\begin{definition}\label{def:induced}
Let $K_r$ be a clique on vertex set $U$.
An edge-weighting $w$ on $K_r$ is \emph{vertex-induced} if there exists a function $a:U\to \mathbb{R}_{\ge 0}$ such that
\begin{equation}\label{eq:induced}
w(uv)=\frac{a(u)+a(v)}{2}\qquad\text{for all }uv\in E(K_r).
\end{equation}
In this case, every cycle $C\subseteq K_r$ satisfies $w(C)=\sum_{v\in V(C)} a(v)$.
\end{definition}

\section{Main theorem}

Equipped with the former definitions, we are ready to state our main result.
As mentioned previously, Theorem~\ref{thm:main} is a generalization of Theorem \ref{ZZ} by taking $w\equiv 1$.

\begin{theorem}\label{thm:main}
Let $(G,w)$ be a connected weighted graph on $n$ vertices, with $C_w(e)$ defined as above.
Then
\begin{equation}\label{eq:main}
\sum_{e\in E(G)}\frac{w(e)}{C_w(e)}\le \frac{n-1}{2}.
\end{equation}
\end{theorem}

\begin{proof}
Define an auxiliary weight function $\phi:E(G)\to \mathbb{R}_{>0}$ by
$$
\phi(e):=\frac{w(e)}{C_w(e)}.
$$
If $C$ is a cycle, then $C_w(e)\ge w(C)$ for each $e\in E(C)$, so
\begin{equation}\label{eq:cyclephi}
\phi(C)=\sum_{e\in E(C)}\frac{w(e)}{C_w(e)}\le \sum_{e\in E(C)}\frac{w(e)}{w(C)}=1.
\end{equation}

Let $B$ be the set of bridges of $G$ and let $H_1,\dots,H_k$ be the components of $G-B$.
Then each $H_i$ is $2$-edge-connected (or a single isolated vertex).
For any bridge $b\in B$, $\phi(b)=\frac{w(b)}{2w(b)}=\frac{1}{2}$.

If $H_i$ is an isolated vertex, then $E(H_i)=\emptyset$ and $\phi(H_i)=0=\frac{n_i-1}{2}$, so the bound holds trivially. We henceforth assume $n_i\ge 2$. Fix a component $H_i$, since every cycle in $H_i$ has $\phi$-weight at most $1$ by \eqref{eq:cyclephi}, Theorem \ref{bf} applied to $(H_i,\phi)$ gives
$$
2\phi(H_i)/(n_i-1)\ \le\ \max\{\phi(C):\ C\subseteq H_i\text{ a cycle}\}\ \le\ 1,
$$
so $\phi(H_i)\le (n_i-1)/2$.

Summing over $i$ and adding the bridges yields
$$
\sum_{e\in E(G)}\frac{w(e)}{C_w(e)}
=\sum_{i=1}^k \phi(H_i)+\sum_{b\in B}\phi(b)
\le \sum_{i=1}^k \frac{n_i-1}{2}+\frac{|B|}{2}.
$$
Finally, removing all bridges increases the number of components by $|B|$, so $\sum_{i=1}^k (n_i-1)=n-k$ and $k=|B|+1$, giving
$$
\sum_{i=1}^k (n_i-1)+|B|=n-1,
$$
and hence the right-hand side equals $\frac{n-1}{2}$.
\end{proof}


\section{Extremals and equality}

\subsection{Explicit equality families}

Theorem~\ref{thm:main} is sharp.
Beyond trees and uniform complete graphs, equality holds for a broad class of ``trees of cliques'' with natural weight patterns.

\begin{definition}
A connected graph $G$ is a \emph{block graph} if every \emph{block} of $G$ (that is, every maximal connected subgraph with no cut-vertex) is a clique.
Equivalently, any two maximal cliques intersect in at most one vertex, and every cycle of $G$ is contained in a unique clique.
\end{definition}

\begin{proposition}\label{prop:induced}
Let $K_r$ be a complete graph with $r\ge 4$ and let $w$ be a vertex-induced weighting in the sense of Definition~\ref{def:induced}.
Then $(K_r,w)$ attains equality in \eqref{eq:main}.
\end{proposition}

\begin{proof}
Let $a:V(K_r)\to \mathbb{R}_{\ge 0}$ satisfy \eqref{eq:induced}.
For any cycle $C$, $w(C)=\sum_{v\in V(C)} a(v)$, so the maximum cycle weight in $K_r$ equals $A:=\sum_{v\in V(K_r)} a(v)$ and is attained by every Hamilton cycle.
Since every edge lies on a Hamilton cycle, $C_w(e)=A$ for all $e\in E(K_r)$.
Therefore
$$
\sum_{e\in E(K_r)} \frac{w(e)}{C_w(e)}=\frac{w(K_r)}{A}.
$$
On the other hand,
$$
w(K_r)=\sum_{uv\in E(K_r)}\frac{a(u)+a(v)}{2}
=\frac12\sum_{v\in V(K_r)} (r-1)a(v)
=\frac{r-1}{2}\,A.
$$
Hence the ratio equals $\frac{r-1}{2}$, which is the right-hand side of \eqref{eq:main} for $n=r$.
\end{proof}

\begin{proposition}\label{prop:block}
Let $G$ be a block graph on $n$ vertices.
Assume that every clique block of size $r\ge 4$ carries a vertex-induced weighting (with respect to its own vertex weights; triangle blocks may have arbitrary positive weights).
Then $(G,w)$ attains equality in \eqref{eq:main}.
\end{proposition}

\begin{proof}
We sum \eqref{eq:main} block-by-block.
If $e$ is a bridge, then $\frac{w(e)}{C_w(e)}=\frac{1}{2}$.

If $e$ lies in a triangle block $T$, then $T$ is the unique cycle containing $e$, so $C_w(e)=w(T)$ and
$$
\sum_{e\in E(T)} \frac{w(e)}{C_w(e)}=\frac{w(T)}{w(T)}=1=\frac{3-1}{2}.
$$

If $e$ lies in a clique block $K_r$ with $r\ge 4$ and the restriction of $w$ to this block is vertex-induced, then Proposition~\ref{prop:induced} shows that this block contributes $(r-1)/2$.

Since $G$ is a block graph, its clique blocks can be ordered $B_1,\dots,B_t$ so that for each $j\ge 2$ the block $B_j$ intersects the union of the previous blocks in exactly one vertex.
Consequently,
$$
n=|V(B_1)|+\sum_{j=2}^t \bigl(|V(B_j)|-1\bigr),
\qquad\text{so}\qquad
n-1=\sum_{j=1}^t \bigl(|V(B_j)|-1\bigr).
$$
Summing the block contributions computed above yields $\frac{n-1}{2}$.
\end{proof}

\subsection{Necessary conditions for equality}

\begin{proposition}\label{prop:nec}
Let $(G,w)$ be a connected weighted graph on $n$ vertices.
If equality holds in \eqref{eq:main}, then for each component $H$ of $G-B$ with at least one edge (where $B$ is the set of bridges) the auxiliary weighting $\phi(e)=\frac{w(e)}{C_w(e)}$ satisfies:

\item $({\rm i})$  $\phi(H)=\frac{|V(H)|-1}{2}$;

\item $({\rm ii})$ there exists a cycle $C\subseteq H$ with $\phi(C)=1$;

\item $({\rm iii})$ for every edge $e\in E(C)$ one has $C_w(e)=w(C)$.
\\

In particular, each bridgeless component is extremal for the Bondy--Fan heavy-cycle theorem under the weighting $\phi$.
\end{proposition}

\begin{proof}
In the proof of Theorem~\ref{thm:main} we obtained for each bridgeless component $H$ with $n_H=|V(H)|$ the bound $\phi(H)\le \frac{n_H-1}{2}$.
If equality holds in \eqref{eq:main}, then equality must hold for every such component, giving (i).

The Bondy--Fan theorem yields a cycle $C$ in $(H,\phi)$ with weight at least $\frac{2\phi(H)}{n_H-1}=1$; but every cycle has weight at most $1$ by \eqref{eq:cyclephi}, so $\phi(C)=1$, proving (ii).

For (iii), note that for any cycle $C$ we have $C_w(e)\ge w(C)$ for each $e\in E(C)$, hence
$$
\phi(C)=\sum_{e\in E(C)} \frac{w(e)}{C_w(e)} \le \sum_{e\in E(C)}\frac{w(e)}{w(C)} = 1.
$$
If $\phi(C)=1$, then equality must hold termwise.
Since all $w(e)>0$, this forces $C_w(e)=w(C)$ for each $e\in E(C)$.
\end{proof}

\subsection{Complete graphs}

On complete graphs one can say more: equality forces the \emph{self-normalized} weights $\phi(e)=\frac{w(e)}{C_w(e)}$ to be induced in the sense of Definition~\ref{def:induced}.

\begin{lemma}\label{lem:hamdecomp}
Let $r\ge 3$.
Let $\mathcal{H}$ be the set of undirected Hamilton cycles in $K_r$.
Then each edge of $K_r$ lies in exactly $(r-2)!$ members of $\mathcal{H}$.
Consequently,
\begin{equation}\label{eq:fractional-ham}
\mathbf{1}_{E(K_r)}=\sum_{H\in \mathcal{H}} \frac{1}{(r-2)!}\,\mathbf{1}_{E(H)}.
\end{equation}
\end{lemma}

\begin{proof}
The number of undirected Hamilton cycles in $K_r$ is $\frac{(r-1)!}{2}$.
By symmetry, each edge lies in the same number $N$ of Hamilton cycles.
Double counting edge--cycle incidences gives
$$
N\binom{r}{2}=r\cdot\frac{(r-1)!}{2},
$$
so $N=(r-2)!$.
Equation \eqref{eq:fractional-ham} is then immediate.
\end{proof}

\begin{lemma}\label{lem:ham-implies-induced}
Let $r\ge 4$ and let $\phi$ be a positive edge-weighting of $K_r$.
If every Hamilton cycle of $K_r$ has the same $\phi$-weight, then $\phi$ is induced: there exists $a:V(K_r)\to\mathbb{R}$ such that $\phi(uv)=\frac{a(u)+a(v)}{2}$ for all edges $uv$.
\end{lemma}

\begin{proof}
Fix distinct vertices $x,y$.
We claim that for any two distinct vertices $u,v\notin\{x,y\}$,
\begin{equation}\label{eq:4cycle}
\phi(uv)+\phi(xy)=\phi(ux)+\phi(vy)=\phi(uy)+\phi(vx).
\end{equation}
Indeed, choose a Hamilton cycle $H$ that contains the edges $ux$ and $vy$ and in which the four vertices $u,x,v,y$ appear in the cyclic order $u-x-\cdots-v-y-\cdots-u$.
Let $H'$ be obtained from $H$ by the standard $2$-opt swap replacing the edges $ux$ and $vy$ by $uv$ and $xy$ (this is possible in $K_r$ and yields another Hamilton cycle).
The two Hamilton cycles share all other edges, so the assumption that $\phi(H)=\phi(H')$ implies $\phi(ux)+\phi(vy)=\phi(uv)+\phi(xy)$.
The second equality in \eqref{eq:4cycle} follows analogously (swap $uy$ and $vx$).

Now define $a:V(K_r)\to \mathbb{R}$ by
$$
a(t):=\phi(tx)+\phi(ty)-\phi(xy)\qquad (t\in V(K_r)).
$$
For distinct $s,t$, using \eqref{eq:4cycle} twice we obtain
$$
a(s)+a(t)
=\bigl(\phi(sx)+\phi(sy)-\phi(xy)\bigr)+\bigl(\phi(tx)+\phi(ty)-\phi(xy)\bigr)
=2\phi(st),
$$
so $\phi(st)=\frac{a(s)+a(t)}{2}$ as required.
\end{proof}

\begin{proposition}\label{prop:Kcomplete}
Let $r\ge 4$ and let $(K_r,w)$ be a positively weighted complete graph.
If equality holds in \eqref{eq:main}, then the self-normalized weights
$$
\phi(e)=\frac{w(e)}{C_w(e)}
$$
form an induced edge-weighting of $K_r$.
Equivalently, there exists $a:V(K_r)\to\mathbb{R}$ such that $\phi(uv)=\frac{a(u)+a(v)}{2}$ for all $uv$.
\end{proposition}

\begin{proof}
Let $\phi(e)=\frac{w(e)}{C_w(e)}$.
As in \eqref{eq:cyclephi}, every cycle $C$ satisfies $\phi(C)\le 1$.
Since equality holds in \eqref{eq:main}, we have $\phi(K_r)=\frac{r-1}{2}$.

Apply \eqref{eq:fractional-ham}.
Multiplying by $\phi$ and summing gives
$$
\phi(K_r)
=\sum_{H\in \mathcal{H}}\frac{1}{(r-2)!}\,\phi(H)
\le \sum_{H\in \mathcal{H}}\frac{1}{(r-2)!}
=\frac{|\mathcal{H}|}{(r-2)!}
=\frac{r-1}{2}.
$$
All inequalities must be equalities, so $\phi(H)=1$ for every Hamilton cycle $H$.
In particular, all Hamilton cycles have the same $\phi$-weight, and Lemma~\ref{lem:ham-implies-induced} yields that $\phi$ is induced.
\end{proof}

\begin{lemma}\label{lem:ham-connected}
Let $r\ge 4$ and let $\mathcal{H}$ be the set of undirected Hamilton cycles of $K_r$.
For any $H,H'\in\mathcal{H}$ there exists a sequence $H=H_0,H_1,\dots,H_t=H'$ such that $H_{i-1}$ and $H_i$ share at least one edge for every $i$.
\end{lemma}

\begin{proof}
Let the vertices of $K_{r}$ be $1,\dots,r$. Every Hamilton cycle can be identified with a cyclic permutation of the vertex set. It is known that the symmetric group $S_{r}$ is generated by all adjacent transpositions $(i,i+1)$ (where $i = 1,\dots,r-1$). Consequently, for the cycles $H$ and $H^{\prime}$ corresponding to the cyclic permutations $\pi$ and $\pi^{\prime}$ respectively, there exists a sequence of adjacent transpositions $s_{1},s_{2},\dots,s_{k}$ such that $\pi^{\prime} = \pi \circ s_{1} \circ s_{2} \circ \dots \circ s_{k}$.

We now define a sequence of cycles $H_{0},H_{1},\dots,H_{k}$ as follows: set $H_{0} = H$, and let $H_{i}$ be the cycle corresponding to the permutation $\pi \circ s_{1} \circ \dots \circ s_{i}$, in other words, $H_{i}$ is obtained from $H_{i-1}$ by applying the adjacent transposition $s_{i}$. Applying an adjacent transposition $(j,j+1)$ corresponds to swapping the two adjacent vertices $j$ and $j+1$ in the cyclic order of the cycle.

Consider any consecutive pair $H_{i-1}$ and $H_{i}$ in this sequence. They are related by a single 2-opt swap. Under such a swap, only four edges those incident to the two swapped vertices and their immediate neighbors can change; all other edges of the cycle remain unchanged. Since $r \ge 4$, the cycle has at least four vertices, and therefore $H_{i-1}$ and $H_{i}$ must share at least one edge outside the local neighborhood affected by the swap.

Thus the constructed sequence $H = H_{0}, H_{1}, \ldots, H_{k} = H^{\prime}$ satisfies the required property.
\end{proof}

\begin{proposition}\label{prop:Kchar}
Let $r\ge 4$ and let $(K_r,w)$ be a positively weighted complete graph.
Then equality holds in \eqref{eq:main} if and only if $w$ is vertex-induced in the sense of Definition~\ref{def:induced}.
In this case, every edge has the same localization denominator, namely
$$
C_w(e)=\sum_{v\in V(K_r)} a(v)\qquad\text{for all }e\in E(K_r).
$$
\end{proposition}

\begin{proof}
If $w$ is vertex-induced, then Proposition~\ref{prop:induced} gives equality.

Conversely, assume equality holds in \eqref{eq:main}.
Set $\phi(e)=\frac{w(e)}{C_w(e)}$.
By Proposition~\ref{prop:Kcomplete}, $\phi$ is induced.
Moreover, the proof of Proposition~\ref{prop:Kcomplete} shows that every Hamilton cycle $H$ satisfies $\phi(H)=1$.

Fix a Hamilton cycle $H$.
Since $C_w(e)\ge w(H)$ for all $e\in E(H)$, we have
$$
1=\phi(H)=\sum_{e\in E(H)}\frac{w(e)}{C_w(e)}\le \sum_{e\in E(H)}\frac{w(e)}{w(H)}=1.
$$
Thus equality holds, forcing $C_w(e)=w(H)$ for every $e\in E(H)$.

If two Hamilton cycles share an edge $e$, then the above shows their total $w$-weights are both equal to $C_w(e)$, hence are equal.
By Lemma~\ref{lem:ham-connected}, all Hamilton cycles have the same total weight; denote it by $W$.
Since every edge of $K_r$ lies on a Hamilton cycle, it follows that $C_w(e)=W$ for all edges $e$.
Therefore $w(e)=W\,\phi(e)$ for all $e$.
Since $\phi$ is induced, so is $w$.

Finally, let $a:V(K_r)\to\mathbb{R}$ be vertex weights inducing $w$.
For any cycle $C$ in $K_r$ we have $w(C)=\sum_{v\in V(C)}a(v)$.
If some vertex $x$ had $a(x)<0$, then a cycle on $V(K_r)\setminus\{x\}$ would have weight
$\sum_{v\ne x}a(v)=\sum_{v}a(v)-a(x)>\sum_{v}a(v)$, contradicting the fact that $W$ is the maximum cycle weight.
Hence $a(v)\ge 0$ for all $v$, completing the proof.
\end{proof}

\begin{corollary}\label{cor:block-char}
Let $G$ be a block graph on $n$ vertices.
If equality holds in \eqref{eq:main}, then every clique block $K_r$ with $r\ge 4$ carries a vertex-induced weighting (triangle blocks may have arbitrary positive weights).
\end{corollary}

\begin{proof}
Let $B$ be a clique block of size $r\ge 4$.
Every cycle containing an edge of $B$ is contained in $B$, so the values $C_w(e)$ for $e\in E(B)$ are computed entirely inside $B$.
Apply Theorem~\ref{thm:main} to $(B,w|_{E(B)})$ to get
$$
\sum_{e\in E(B)}\frac{w(e)}{C_w(e)}\le \frac{|V(B)|-1}{2}.
$$
Summing this bound over all blocks of $G$ yields exactly the global bound $\frac{n-1}{2}$.
Indeed, for block graphs one has
$$
n-1=\sum_{\text{blocks }B} (|V(B)|-1),
$$
as in the proof of Proposition~\ref{prop:block}.
Thus if equality holds in \eqref{eq:main} for $G$, then equality must hold in the displayed inequality for every block $B$.
Proposition~\ref{prop:Kchar} then implies that $w|_{E(B)}$ is vertex-induced.
\end{proof}

\begin{remark}[On a full equality characterization]
Bondy and Fan~\cite{BF} completely characterize the $2$-edge-connected weighted graphs for which their heavy-cycle bound is tight.
Proposition~\ref{prop:nec} shows that an equality characterization for \eqref{eq:main} can in principle be obtained by applying that characterization to the self-normalized weighting $\phi$ and keeping track of the ``tight cycle'' condition~(iii).
Proposition~\ref{prop:Kchar} gives a clean instance of this program: on complete graphs, the original weights themselves must be vertex-induced.
\end{remark}

\section{Concluding remarks}
In this note, we have established a weighted local analogue of the Bondy–Fan theorem, and shown that the resulting inequality is sharp by identifying a broad family of extremal graphs. 
The proof relies on a simple but effective self-normalization device: replacing $w$ by the auxiliary weighting $\phi(e)=\frac{w(e)}{C_w(e)}$ forces every cycle to carry $\phi$-weight at most $1$, after which the Bondy–Fan theorem applies directly. This reduction highlights a clean correspondence between the global and local perspectives on long cycles in weighted graphs.

The inequality in Theorem~\ref{thm:main} yields several useful consequences when one imposes a threshold on the value $C_w(e)$. Edges with small $C_w(e)$, those lying only on relatively light cycle, cannot collectively carry too much weight, as made precisely by the following estimates.

\begin{corollary}\label{cor:threshold}
Let $(G,w)$ be a connected weighted graph on $n$ vertices.
For any $T>0$,
\begin{equation}\label{eq:threshold}
\sum_{e\in E(G):\, C_w(e)\le T} w(e)\ \le\ \frac{T(n-1)}{2}.
\end{equation}
\end{corollary}

\begin{proof}
For edges with $C_w(e)\le T$ we have $w(e)/C_w(e)\ge w(e)/T$.
Summing and applying \eqref{eq:main} gives
$\frac1T\sum_{C_w(e)\le T} w(e)\le \frac{n-1}{2}$.
\end{proof}

\begin{corollary}\label{cor:sup}
Let $(G,w)$ be a connected weighted graph on $n$ vertices.
For any $T>0$,
$$
\sum_{e:\,C_w(e)>T} w(e)
\ \ge\ w(G)-\frac{T(n-1)}{2}.
$$
\end{corollary}

\begin{proof}
This is immediate from \eqref{eq:threshold} by subtracting from $w(G)=\sum_{e\in E(G)}w(e)$.
\end{proof}

\begin{corollary}\label{cor:forest}
Let $(G,w)$ be a connected weighted graph.
Define
$$
F:=\{e\in E(G):\ C_w(e)<2w(e)\}.
$$
Then $F$ is acyclic (in particular $|F|\le n-1$).
\end{corollary}

\begin{proof}
If $F$ contained a cycle $C$, then for any $e\in E(C)$ we would have $w(C)\le C_w(e)<2w(e)$.
Choosing two distinct edges $e,f\in E(C)$ and using $w(C)\ge w(e)+w(f)$ yields $w(f)<w(e)$ and $w(e)<w(f)$, a contradiction.
\end{proof}

Several natural questions remain open. 
First, a complete characterization of equlity in Theorem~\ref{thm:main} for general weighted graphs, analogous to the characterization of Bondy and Fan, has not yet been obtained. Proposition~\ref{prop:nec} reduces this problem to understanding when the Bondy-Fan bound is tight under the self-normalized wighting $\phi$, but a full combinatorial description is still missing. Second, it would be interesting to explore whether similar self-normalization arguments can be applied to other extremal parameters, such as the length of the longest path or the size of a maximum matching, thereby yielding further localized inequalities in the spirit of~\cite{LN,ZZ}.

\end{document}